\documentclass[12pt,oneside,english,a4paper]{amsart}
\usepackage[T1]{fontenc}
\usepackage[utf8]{inputenc}
\usepackage[dvipsnames]{xcolor}
\usepackage[yyyymmdd,hhmmss]{datetime}
\usepackage[numbers,sort,compress]{natbib}
\usepackage{comment,todonotes}

\usepackage{bbm}
\usepackage{amstext,amsthm,amssymb}
\usepackage{xparse,mathrsfs,mathtools}
\usepackage[english]{babel}
\usepackage[pdftex,unicode=true,%
pdfusetitle,pdfdisplaydoctitle=true,%
pdfpagemode=UseOutlines,%
bookmarks=true,bookmarksnumbered=false,bookmarksopen=true,bookmarksopenlevel=2,%
breaklinks=true,%
pdfencoding=unicode,psdextra,
pdfcreator={},
pdfborder={0 0 0}
]{hyperref}
\usepackage{lmodern,microtype}
\usepackage{geometry}
\usepackage[capitalise,nameinlink]{cleveref}

\newtheorem{thm}{Theorem}[section]
\newtheorem{lem}[thm]{Lemma}
\newtheorem{cor}[thm]{Corollary}
\newtheorem{prop}[thm]{Proposition}
\theoremstyle{remark}
\newtheorem{rem}[thm]{Remark}

\numberwithin{equation}{section}

\newcommand{\Starphi}{S_\phi^*}

\newcommand{\supp}{\operatorname{supp}}

\title[Selberg's Central Limit Theorem weighted by Linear Statistics of Zeta Zeros]{Selberg's Central Limit Theorem weighted \\ by Linear Statistics of Zeta Zeros}

\author{Alessandro Fazzari}
\address{D\'epartement de math\'ematiques et de statistique, Universit\'e de Montr\'eal. CP 6128, succ. Centre-ville. Montreal, QC H3C 3J7, Canada}
\email{alessandro.fazzari@umontreal.ca}

\author{Maxim Gerspach}
\address{Department of Mathematics, University of Goettingen, Bunsenstr. 3-5, 37073 Goettingen, Germany}
\email{maxim.gerspach@mathematik.uni-goettingen.de}

\author{Paolo Minelli}
\address{Institute for Analysis and Number Theory, TU-Graz, Steyrergasse 30, 8010 Graz, Austria}
\email{minelli@tugraz.at}

\date{\today}

\begin{document}

\begin{abstract}
    We consider the value distribution of the logarithm of the Riemann zeta function on the critical line, weighted by the local statistics of zeta zeros. We show that, with appropriate normalization, it satisfies a complex Central Limit Theorem, provided that the Fourier support of the test function in the linear statistics is sufficiently small. 
    For the imaginary part, we extend this support condition up to its natural barrier under the Riemann Hypothesis. Finally, we prove that the correlation between $\log \zeta$ and the one-level density, while negligible on the level of Selberg's Central Limit Theorem, only decays at a rather slow rate if the Riemann Hypothesis is assumed. Our results can be viewed as a combination of Selberg's Central Limit Theorem with work of Hughes and Rudnick on mock-Gaussian behavior of the local statistics.
\end{abstract}

\maketitle

\section{Introduction}

Selberg's celebrated Central Limit Theorem states that the logarithm of the Riemann zeta function on the critical line, when suitably normalized, is approximately distributed like a complex Gaussian random variable. Specifically, if $\tau$ is sampled uniformly on an interval $[T,2T]$, we have convergence in distribution
\begin{equation}\label{5june.1}
    \frac{\log \zeta ( \tfrac12 + i \tau)}{\sqrt{\log \log T}} \overset{T \to \infty}{\longrightarrow} \mathcal{N}_{\mathbb{C}}(0,1),
\end{equation}
where $\mathcal{N}_{\mathbb{C}}(0,1)$ denotes 
a standard complex Gaussian. 
The works of Selberg \cite{Selberg, Selberg2} and Tsang \cite{Tsang1} provide full proofs of the Central Limit Theorem for the real and imaginary parts of log-zeta individually. A sketch of the proof of \eqref{5june.1} in the complex case can be found in \cite{Joyner}; see also \cite{Laurinchikas, BombieriHejhal} for further discussion and generalizations.

The classical approach to proving such a theorem is based on the method of moments. To establish \eqref{5june.1}, it suffices to show that for any fixed nonnegative integers $h,\ell$, we have
\begin{align*}
    \frac{1}{T} \int_T^{2T} \left(\log\zeta(\tfrac12+it)\right)^h  \left(\overline{\log\zeta(\tfrac12+it)}\right)^\ell dt
    =  \mathbbm{1}(h=\ell) \times h!(\log\log T)^h  + O((\log\log T)^{(h+\ell-1)/2}).
\end{align*}
More recent proofs of Selberg's Central Limit Theorem establish a convergence in distribution for the real part of $\log \zeta$ without requiring the computation of all moments; see \cite{BombieriHejhal,Laurinchikas2} and the more recent approach proposed by Radziwi{\l}{\l} and Soundararajan \cite{RadziwillSound}.

In modern terminology, a revealing way to understand why Selberg's Central Limit Theorem holds is via the so-called hybrid model for zeta. In \cite{GonekHughesKeating}, Gonek, Hughes, and Keating derive a formula that, under the Riemann Hypothesis, suggests an approximation of the type
\begin{align}\label{HybridModel}
    \zeta(\tfrac12+it) \approx 
    \prod_{p\leq X} \bigg(1-\frac{1}{p^{1/2+it}} \bigg)^{-1}
    \times \prod_{|\gamma-t|\leq \frac{1}{\log X}} \bigg(1-\frac{\frac12+it}{\frac12+i\gamma}\bigg)
\end{align}
where $\frac12+i\gamma$ runs over the nontrivial zeros of zeta. Thus, \eqref{HybridModel} expresses zeta as a product of a truncated Euler product and a truncated Hadamard product, with $X$ serving as a parameter that controls the truncation ranges. Upon taking the logarithm, the Gaussian nature emerges: the logarithm of the Euler product in \eqref{HybridModel} essentially consists of a sum of random variables $p^{-i\tau}$, weighted by $p^{-1/2}$. Since $p$ ranges over primes, the variables $p^{-i\tau}$ behave like independent, uniformly distributed points on the unit circle, and thus their sum tends toward a complex Gaussian by the Central Limit Theorem. The bulk of Selberg's work is then to show that the contribution from the zeros in \eqref{HybridModel} is negligible on average over $t$.

In light of \eqref{HybridModel}, primes alone do not capture the full nature of zeta; to reflect its complete structure, one also needs an understanding of its nontrivial zeros.
The seminal works of Montgomery \cite{Montgomery1} and Rudnick–Sarnak \cite{RudnickSarnak2} study the level correlations of zeros and reveal a striking connection with random matrix theory. The correlations of the zeros of zeta are conjectured to match those of the eigenvalues of large random unitary matrices. This phenomenon has been proven (partly under the Riemann Hypothesis) for a restricted class of test functions appearing in the definition of these correlations; see \cite{Montgomery1, RudnickSarnak}.

In \cite{HughesRudnick}, Hughes and Rudnick investigate the linear statistics of the nontrivial zeros of zeta. More precisely, 
denoting the zeros by $\frac12+i\gamma$ with $\gamma\in\mathbb C$
, they consider the centered moments of the quantity
$$ N_{\phi}(t) := \sum_{\gamma} \phi \left( \frac{\log T}{2\pi}(\gamma - t) \right), $$
where $\phi$ is a suitable test function with compactly supported Fourier transform.
One can view $N_\phi$ as a smooth
version of the counting function of zeros in intervals of length comparable to the mean spacing. Hughes and Rudnick prove that the first few moments of $N_\phi$ are asymptotic to those of a Gaussian random variable with mean $\hat\phi(0)$ and variance
\begin{align}\label{3june.1}
    \sigma_\phi^2 := \int_{-\infty}^{+\infty} \min(|u|,1) \hat\phi(u)^2 du.
\end{align}
Specifically, provided that $\hat\phi$ is supported in $(-2/m,2/m)$, they show that
\begin{align}\label{HRtheorem}
    \frac{1}{T} \int_{-\infty}^{+\infty} (N_\phi(t)-\hat\phi(0))^k \; \omega \left(\frac{t-T}{T} \right) dt 
    = \mu_k \sigma_\phi^{k} + o(1),
\end{align}
for every $k\leq m$. Here, $\omega$ is a nonnegative weight function of total mass 1 whose Fourier transform is compactly supported, and
\begin{align}\label{3june.2}
\mu_k := \begin{cases} (k-1)!! & \text{if } k \text{ even} \\ 0 & \text{if } k \text{ odd} \end{cases}
\end{align}
is the $k$-th moment of a standard real Gaussian. Note that the smoothing function $\omega$ in \eqref{HRtheorem} is not strictly necessary; the same result holds with the sharp cutoff $\omega = \mathbbm{1}_{[0,1]}$, as implied by \cref{mainthm} and \cref{RemarkHRsharp}.
According to \eqref{HRtheorem}, the first $m$ centered moments of $N_\phi$ agree with those of a Gaussian. The analogous result has also been proven in the random matrix theory setting by Hughes and Rudnick \cite{HughesRudnick2}, where they further show that higher moments are not Gaussian. 
This phenomenon, which is also expected for the Riemann zeta function, is referred to as mock-Gaussian behavior.

In this paper, we investigate the joint distribution of the logarithm of the Riemann zeta function and the linear statistics of its nontrivial zeros, thereby establishing a joint version of Selberg's Central Limit Theorem and Hughes-Rudnick's mock-Gaussian result. In other words, we study how the statistics of zeros near $t$ are influenced by weighting by $\log\zeta(\frac12+it)$. This captures the distribution of zeros around typical points, i.e. those governing the value distribution of log-zeta.

To this end, we adopt a setup similar to that in \cite{HughesRudnick}. Throughout the whole paper, we fix a function $\phi$ that satisfies the following conditions:
\begin{equation}\begin{split}\label{24may.1}
& \phi \text{ is even, real-valued on the real line, and not identically zero;} \\ 
& \hat \phi \text{ is smooth and compactly supported, with } \eta = \eta_\phi :=\inf\{a>0: \supp \hat\phi \subseteq (-a,a)\}. 
\end{split}\end{equation}
Note that, under these assumptions, $\hat\phi$ is real-valued on the real line.

Our main result provides an asymptotic formula for the joint moments of $\log\zeta$, its complex conjugate, and $N_\phi$.

\begin{thm}\label{mainthm}
    Let $T \ge 3$ be a (large) parameter and let $h,\ell,k$ be fixed nonnegative integers with $k$ even. 
    For any function $\phi$ that satisfies \eqref{24may.1} with $\supp \hat\phi \subseteq (-\frac{2}{k+2},\frac{2}{k+2})$, we have
    \begin{equation}\begin{split}\label{3june.3}
        \frac{1}{T} \int_T^{2T} &\left(\log\zeta(\tfrac{1}{2}+it)\right)^h \left(\overline{\log\zeta(\tfrac{1}{2}+it)} \right)^\ell \left(N_\phi(t)-\hat\phi(0)\right)^k dt \\
        &= \mathbbm{1}(h = \ell) \times \mu_k \sigma_\phi^k h! \left(\log\log T\right)^{h} + O \left( (\log\log T)^{(h + \ell -1)/2} \right),
    \end{split}\end{equation} 
    where $\sigma_\phi^2$ and $\mu_k$ are defined as in \eqref{3june.1} and \eqref{3june.2}, respectively.
\end{thm}

The support condition required in \cref{mainthm} is not quite optimal in the sense of Hughes and Rudnick \cite{HughesRudnick}. While they showed that the $k$-th centered moment of $N_\phi$ is Gaussian for $\eta< 2/k$, our result holds only under the slightly more restrictive condition $\eta<2/(k+2)$. As further discussed in \cref{SectionApproximations}, this limitation arises from repeated applications of H\"older's inequality when approximating $\log\zeta$ and $N_\phi$ by suitable Dirichlet polynomials. Assuming the Riemann Hypothesis, we can improve the above result and obtain full support $\eta<2/k$ for $\Im \log \zeta$. We refer the reader to \cref{RemarkImvsRe} for a detailed discussion on the distinction between the real and imaginary parts of $\log\zeta$.

\begin{thm}\label{mainthm2}
    Assume the Riemann Hypothesis. Let $T \ge 3$ be a (large) parameter and let $\ell,k$ be fixed nonnegative integers with $k$ even. 
    For any function $\phi$ that satisfies \eqref{24may.1} with $\supp \hat\phi \subseteq (-\frac{2}{k},\frac{2}{k})$, we have
    \begin{equation}\begin{split}\label{3june.5}
        \frac{1}{T} \int_T^{2T} & \left(\Im\log\zeta(\tfrac{1}{2}+it) \right)^\ell \left(N_\phi(t)-\hat\phi(0)\right)^k dt \\
        &= \mu_\ell \mu_k \sigma_\phi^k \left(\log\log T\right)^{\ell/2} + O \left( (\log\log T)^{(\ell -1)/2} \right),
    \end{split}\end{equation} 
    where $\sigma_\phi^2$ and $\mu_k$ are defined as in \eqref{3june.1} and \eqref{3june.2}, respectively.
\end{thm}

From the moment estimate provided by \cref{mainthm} (respectively, \cref{mainthm2}), we can infer a convergence in distribution of $\log \zeta$ (respectively, $\Im \log \zeta$) with respect to the weighted measure introduced by the local zero statistics. We note that, as a consequence of \eqref{24may.1} and the explicit formula given by \cref{ExplForm_SumOverZeros} below, the expression $N_\phi(t) - \hat \phi(0)$ is approximately real (and it is in fact real under the Riemann Hypothesis). For any even integer $k$, we thus define the probability measure
\begin{align}\label{3june.4}
d\mu_{\phi,k}(t) := | N_\phi(t)-\hat\phi(0)|^{k} \frac{dt}{C_{\phi,k}}
\end{align}
on $[T,2 T]$, where $C_{\phi,k} = C_{\phi,k}(T) \sim \mu_k \sigma_\phi^k$ is the appropriate normalizing constant. By the method of moments, one quickly obtains the following distributional statements.

\begin{cor}\label{cor1}
Let $k$ be an even nonnegative integer. Suppose that $\phi$ satisfies \eqref{24may.1} with $\supp \hat\phi \subseteq (-\frac{2}{k+2},\frac{2}{k+2})$. For $T \ge 3$, let $\tau_k$ be a random variable sampled according to the probability measure $d\mu_{\phi,k}(t)$ on $[T,2 T]$. Then we have the convergence in distribution
\begin{equation}\label{10june.1}\notag
    \frac{\log \zeta ( \tfrac12 + i \tau_k)}{\sqrt{\log \log T}} \overset{T \to \infty}{\longrightarrow} \mathcal{N}_{\mathbb{C}}(0,1).
\end{equation}
\end{cor}

\begin{cor}\label{cor2}
Assume the Riemann Hypothesis. Let $k$ be an even nonnegative integer and suppose that $\phi$ satisfies \eqref{24may.1} with $\supp \hat\phi \subseteq (-\frac{2}{k},\frac{2}{k})$. For $T \ge 3$, let $\tau_k$ be a random variable sampled according to the probability measure $d\mu_{\phi,k}(t)$ on $[T,2 T]$. Then we have the convergence in distribution
\begin{equation}\label{10june.2}\notag
    \frac{\Im \log \zeta ( \tfrac12 + i \tau_k)}{\sqrt{\frac{1}{2} \log \log T}} \overset{T \to \infty}{\longrightarrow} \mathcal{N}(0,1).
\end{equation}
\end{cor}

The two corollaries above are examples of weighted Central Limit Theorems, where the weight is arithmetic in nature. Another type of weight studied in the literature arises from moments of zeta itself. This topic, which is intimately related to the large deviation regime \cite{Radziwill1,ArguinBailey,ArguinBailey2}, has been studied both for zeta \cite{FazzariWCLT1, FazzariWCLT2} and, in mollified form, for more general $L$-functions \cite{BELP}. These works indicate that the value distribution of log-zeta around points $t$ such that $\zeta(\frac12+it)\asymp (\log T)^k$, i.e. those responsible for the $2k$-th moment, is approximately Gaussian with the same variance but mean $k\log\log T$.

In recent years, the joint distribution of the correlations of zeros of zeta and its moments has also been investigated \cite{FazzariW1LD,BettinFazza}. These ideas, which can be used to deduce results on large values of $\zeta$ near its
zeros, have been extended to more general $L$-functions \cite{FazzariW1LD, Su-Su, Sugiyama,Zhao}, and partially to the case of the 2-correlation \cite{FazzariWPC}.

Corollaries \ref{cor1} and \ref{cor2} show that $\log\zeta$ and $N_\phi$ behave approximately like two independent random variables. Indeed, their joint moments factor into the product of the Gaussian moments of the logarithm of zeta and the mock-Gaussian moments of the correlation of zeros. In this sense, our results can be viewed as a joint version of Selberg's Central Limit Theorem and the work of Hughes and Rudnick, reflecting the asymptotic independence of these two objects.

While true at leading order, these considerations on the independence of $\log\zeta$ and $N_\phi$ break down at a finer level. This is due to the fact that secondary terms in Selberg's Central Limit Theorem are significantly more delicate, as contributions from the zeros in \eqref{HybridModel} begin to play a nonnegligible role; see, for instance, \cite{Goldston1,LMQ,FG}.
To illustrate this, we consider the correlation of $\log\zeta(\frac12+it)$ and $N_\phi(t)-\hat\phi(0)$. Individually, the first moment of $\log\zeta$ is extremely small; under the Riemann Hypothesis, a minor adaptation of the proof of \eqref{GoldstonEst} shows that 
\begin{align*}
\frac{1}{T}\int_T^{2T} \log\zeta(\tfrac12+it)dt \ll \frac{\log T}{T}.
\end{align*}
Similarly, we will see that
\begin{align*} 
\frac{1}{T} \int_T^{2 T} \left (N_\phi(t) - \hat \phi(0) \right) \, dt \ll \frac{1}{\log T}, 
\end{align*}
for example as a consequence of \cref{RemarkHRsharp}.
For the correlation between the two quantities, we have the following estimate.

\begin{prop}\label{PropCorrelation}
    Assume the Riemann Hypothesis. Let $T \ge 3$ be a (large) parameter and let $\phi$ be a function satisfying \eqref{24may.1} with $\supp \hat\phi \subseteq (-1,1)$. Then we have
    \begin{equation}\notag
        \frac{1}{T}\int_{T}^{2T} \log \zeta(\tfrac{1}{2}+it) \, \left(N_\phi(t) - \hat\phi(0)\right) dt 
        = -\frac{\phi(0)}{2}+O\bigg(\frac{\sqrt{\log\log T}}{\log T}\bigg).
    \end{equation}
\end{prop}
Another way of stating this is by looking at the correlation coefficient between the random variables $\log \zeta(\tfrac{1}{2}+i \tau)$ and $N_\phi(\tau)$ with $\tau$ uniformly distributed on $[T,2T]$. Namely, \cref{PropCorrelation} immediately implies (under the Riemann Hypothesis) that
\[ \frac{\mathbb{E}\big[\log \zeta(\tfrac{1}{2}+i \tau) ( N_\phi(\tau) - \hat \phi(0) )\big]}{\mathbb{E}\big[ | \log \zeta(\tfrac{1}{2}+i \tau) |^2 \big]^{1/2} \mathbb{E} \big[( N_\phi(\tau) - \hat \phi(0) )^2\big]^{1/2}} \sim - \frac{\phi(0)}{2 \sigma_\phi } \frac{1}{\sqrt{\log \log T}} \]
as $T \to \infty$, assuming for ease of notation that $\phi (0) \ne 0$.

The paper is structured as follows: \cref{SectionCorrelation} is dedicated to the explicit formula that expresses $N_\phi$ in terms of a sum over prime powers, and to the proof of \cref{PropCorrelation}. As a first step toward Theorems \ref{mainthm} and \ref{mainthm2}, \cref{SectionApproximations} reduces the problem to studying moments of Dirichlet polynomial approximations. These moments are computed in \cref{SectionMomentEstimates}, concluding the paper.

\subsection*{Acknowledgments} 
This work was partially carried out while M.G. and P.M. were based at KTH. A.F. is grateful to Pär Kurlberg for his invitation, and would like to thank Sandro Bettin and Youness Lamzouri for inspiring conversations.
A.F. is a member of the INdAM group GNAMPA, and is supported by the Fonds de recherche du Qu\'ebec - Nature et technologies, Projet de recherche en \'equipe 300951.
P.M. thanks the Knut and Alice Wallenberg foundation grant KAW 2019.0503 from Prof. Kurlberg. P.M. is partially supported by the Austrian Science Fund (FWF),  grant-doi 10.55776/P35322.
This research was funded in whole or in part by the Austrian Science Fund (FWF) 10.55776/PAT4579123. For open access purposes, the authors have applied a CC BY public copyright license to any author accepted manuscript version arising from this submission.
\normalsize


\section{Explicit formula and Proof of \texorpdfstring{\cref{PropCorrelation}}{PropCorrelation}}\label{SectionCorrelation}

First, we state the following version of the explicit formula.

\begin{lem}\label{ExplForm_SumOverZeros}
    Let $T \ge 3$ be a (large) parameter and let $\phi$ be a function satisfying \eqref{24may.1}. Then, for any $t\in[T,2T]$, we have
    \begin{align*}
        N_\phi(t) 
        &= \hat\phi(0) + S_\phi(t) + O\bigg(\frac{1}{\log T}\bigg),
    \end{align*}
    with
    \begin{align*}\notag
        S_\phi(t) &:= -\frac{2}{\log T} \sum_{n=1}^{\infty} \frac{\Lambda(n)}{\sqrt{n}} \hat\phi \bigg(\frac{\log n}{\log T}\bigg) \cos(t\log n).
    \end{align*}
\end{lem}

\begin{proof}
    This is a slight modification of the classical explicit formula. Starting from \cite[Lemma 2.1]{HughesRudnick}, it suffices to follow the steps on page 6 of \cite{BettinFazza}. Specifically, one applies Stirling's approximation formula for the $\Gamma$-functions and repeatedly integrates by parts to handle the polar terms. The claim then follows from the trivial approximation $\log\frac{t}{2\pi}=\log T +O(1)$ for $T\leq t\leq 2T$.
\end{proof}

\begin{rem}\label{Rem_SvsS*}
    For our applications, it is convenient to state a version of the explicit formula in which the sum over prime powers is truncated at squares of primes. To this end, we define
    \begin{align*}\Starphi(t) := -\frac{2}{\log T} \sum_{n=1}^{\infty} \frac{\Lambda^*(n)}{\sqrt{n}} \hat\phi \bigg(\frac{\log n}{\log T}\bigg) \cos(t\log n),
    \end{align*}
    where $\Lambda^*(n)=\log p$ if $n=p^\alpha$ for some prime $p$ and $\alpha \in \{1,2\}$ and $\Lambda^*(n) = 0$ otherwise. With this notation, it is clear that 
    \begin{equation}\label{SS*Comp}
    S_\phi(t) = \Starphi(t) + O\left(\frac{1}{\log T}\right),
    \end{equation}
    and thus \cref{ExplForm_SumOverZeros} can be rewritten as
    \begin{align}\label{ExplFormS*}
    N_\phi(t) = \hat\phi(0) +\Starphi(t)+\bigg(\frac{1}{\log T}\bigg).
    \end{align} 
\end{rem}

Equipped with \cref{ExplForm_SumOverZeros}, we can now estimate the correlation between $\log \zeta$ and $N_\phi$.

\subsection*{Proof of \texorpdfstring{\cref{PropCorrelation}}{PropCorrelation}}
    By Lemma \ref{ExplForm_SumOverZeros}, we have
    \begin{equation}\begin{split}\label{CorrExplForm} \frac{1}{T}\int_{T}^{2T} \log \zeta(\tfrac{1}{2}+it) \, \left(N_\phi(t) - \hat\phi(0)\right) dt &= \frac{1}{T}\int_{T}^{2T} \log \zeta(\tfrac{1}{2}+it) \, S_\phi(t) dt \\ &+ O \left( \frac{1}{T \log T} \int_T^{2 T} | \log \zeta (\tfrac{1}{2}+it) | \, dt \right). \end{split}\end{equation}
    As for the error term, an application of the Cauchy-Schwarz inequality followed by the classical Selberg Central Limit Theorem for the second moment yields
    \begin{equation*}
        \begin{split}
            \frac{1}{T} \int_T^{2 T} | \log \zeta (\tfrac{1}{2}+it) | \, dt &\ll \left( \frac{1}{T} \int_T^{2 T} | \log \zeta (\tfrac{1}{2}+it) |^2 \, dt \right)^{1/2} \ll \sqrt{\log \log T}.
        \end{split}
    \end{equation*}
    Plugging this into \eqref{CorrExplForm}, it thus suffices to show that
    \begin{equation}\label{CorrSTS}
        \frac{1}{T}\int_{T}^{2T} \log \zeta(\tfrac{1}{2}+it) \, S_\phi(t) dt 
        = -\frac{\phi(0)}{2}+O\bigg(\frac{\sqrt{\log\log T}}{\log T}\bigg).
    \end{equation}
    Expanding the cosines in the definition of $S_\phi$ in terms of exponentials, this amounts to understanding the correlation of $\log \zeta$ with $n^{\pm i t}$.
    Work of Goldston \cite[p. 169]{Goldston1} under the Riemann Hypothesis implies that for any integer $2 \le n \le T$  (say), we have
    \begin{equation}\label{GoldstonEst}
        \begin{split}
            \int_T^{2 T} \log \zeta(\tfrac{1}{2}+it) n^{i t} \, dt &= \frac{T \Lambda(n)}{\sqrt{n} \log n} + O \left( \sqrt{n} \log T \right), \\
            \int_T^{2 T} \log \zeta(\tfrac{1}{2}+it) n^{-i t} \, dt &\ll \log T.
        \end{split}
    \end{equation}
    The proof of this statement is by means of a contour shift argument on the functions $\log \zeta(s) n^{\pm s}$ followed by evaluating the associated Dirichlet series in the region of absolute convergence. It is an effective version of previous work \cite{Titchmarsh1} by Titchmarsh.

    As a consequence of \eqref{GoldstonEst}, we obtain
    \begin{equation}\label{CorrInterm}
        \begin{split}
            \frac{1}{T}&\int_{T}^{2T} \log \zeta(\tfrac{1}{2}+it) \, S_\phi(t) dt \\ &= -\frac{1}{T \log T} \sum_{n \ge 2} \frac{\Lambda(n)}{\sqrt{n}} \hat \phi \left( \frac{\log n}{\log T} \right) \int_{T}^{2 T} \log \zeta(\tfrac{1}{2}+it) \left( n^{i t} + n^{-i t} \right) \, dt \\
            &= -\frac{1}{\log T} \sum_{n \ge 2} \frac{\Lambda^2(n)}{n \log n} \hat \phi \left( \frac{\log n}{\log T} \right) + O \bigg( \frac{1}{T} \sum_{n \le T^{\eta}} \Lambda(n) \bigg)
        \end{split}
    \end{equation}
    using the support condition on $\hat \phi$ to restrict the sum over $n$ to be short enough. The error term in \eqref{CorrInterm} is more than acceptable by just using a trivial bound. As for the main term, an elementary summation by parts argument using the Prime Number Theorem shows that we have
    \[ -\frac{1}{\log T} \sum_{n \ge 2} \frac{\Lambda^2(n)}{n \log n} \hat \phi \left( \frac{\log n}{\log T} \right) = - \frac{1}{2} \int_0^{\eta} \hat \phi(x) \, dx + O \left( \frac{1}{\log T} \right) = - \frac{\phi(0)}{2} + O \left( \frac{1}{\log T} \right), \]
    where the last step follows from the support condition on $\hat \phi$ together with the assumption that $\hat \phi$ is even. We therefore deduce \eqref{CorrSTS} and hence the claim. 
\qed

\section{Approximating log-zeta and the sum over zeros}\label{SectionApproximations}
As a first step toward proving Theorems \ref{mainthm} and \ref{mainthm2}, we approximate the moments of $\log\zeta(\frac{1}{2}+it)$ with respect to the weighted measure $(N_\phi(t)-\hat\phi(0))^kdt$ by those of suitable Dirichlet polynomials.
For $N_\phi$ the natural approximation comes from \cref{ExplForm_SumOverZeros}. For $\log\zeta$ we use the following standard Dirichlet polynomial:
\begin{align}\label{4june.1}
P_x(t) := \sum_{p\leq x} \frac{1}{p^{1/2+it}},
\end{align}
where $x$ will be a small (fixed) power of $T$.
The desired approximation is obtained unconditionally in the following proposition, provided that the support of $\hat\phi$ is contained in $(-\tfrac{2}{k+2},\tfrac{2}{k+2})$. 

\begin{prop}\label{PropApprox}
    Let $T \ge 3$ be a (large) parameter and let $h,\ell,k$ be fixed nonnegative integers with $k$ even. Suppose that $\phi$ is a function satisfying \eqref{24may.1} for some $0<\eta<2/(k+2)$. Finally, let $0 < \delta_1 < \delta_2 := \tfrac{1}{100(h+\ell+1)(k+2)}$ be fixed and suppose that $T^{\delta_1} \le x \le T^{\delta_2}$. Then we have
    \begin{align*}
        \frac{1}{T} \int_T^{2T} &(\log\zeta(\tfrac{1}{2}+it))^h (\overline{\log\zeta(\tfrac{1}{2}+it)} )^\ell (N_\phi(t)-\hat\phi(0))^k dt \\
        &= \frac{1}{T}\int_T^{2T} P_x(t)^h \overline{P_x(t)}^\ell \Starphi(t)^k dt 
        +O\left( (\log\log T)^{(h+\ell-1)/2} \right).
    \end{align*}
\end{prop}

\begin{proof}
The approximation of $N_\phi(t)-\hat\phi(0)$ by $\Starphi(t)$ is a standard application of the explicit formula. Indeed, by applying \eqref{ExplFormS*} and expanding the $k$-th power, one has
    \begin{equation}\begin{split}\label{28feb.1}
        \frac{1}{T} \int_T^{2T} &(\log\zeta(\tfrac{1}{2}+it))^h (\overline{\log\zeta(\tfrac{1}{2}+it)})^\ell (N_\phi(t)-\hat\phi(0))^k dt \\
        &= \frac{1}{T}\int_T^{2T} (\log\zeta(\tfrac{1}{2}+it))^h (\overline{\log\zeta(\tfrac{1}{2}+it)})^\ell \Starphi(t)^k dt 
        +O(\mathcal E_1),
    \end{split}\end{equation}
    with 
    \begin{align*} 
        \mathcal E_1 
        \notag \ll \max_{1\leq s\leq k} \frac{1}{(\log T)^s} \frac{1}{T}\int_T^{2T} |\log\zeta(\tfrac{1}{2}+it)|^{h+\ell} |\Starphi(t)|^{k-s} dt .
    \end{align*}
    Since $k$ is even, by applying H\"older's inequality, we bound the above by
        \begin{equation}\begin{split}\label{22may.1}
        \ll \max_{1\leq s\leq k} \frac{1}{(\log T)^s} &\bigg(\frac{1}{T}\int_T^{2T} |\log\zeta(\tfrac{1}{2}+it)|^{2k(h+\ell)} dt \bigg)^{1/2k} 
        \bigg(\frac{1}{T}\int_T^{2T} \Starphi(t)^{k} dt \bigg)^{(k-s)/k}.
    \end{split}\end{equation}
    For the last factor, we simply need to show that the expression is bounded. Let $\omega \ge \tfrac{1}{2} \mathbbm{1}_{[0,1]}$ denote some nonnegative weight function of total mass $1$ with compactly supported Fourier transform. Then we have
    \begin{equation}\label{HRMomentBound}
        \frac{1}{T}\int_T^{2T} \Starphi(t)^{k} dt 
        \le \frac{2}{T} \int_{\mathbb{R}} \Starphi(t)^{k} \omega \left( \frac{t-T}{T} \right) \,dt
        = \frac{2}{T} \int_{\mathbb{R}} S_\phi(t)^{k} \omega \left( \frac{t-T}{T} \right) \,dt + O(1)
        \ll 1,
    \end{equation}
    by invoking \cite[Theorem 2.6]{HughesRudnick}.
    Next, since $|z| \leq |\Re z|+|\Im z|$, Selberg's central limit theorem ensures that the average involving $\log$-zeta in \eqref{22may.1} is bounded by $\ll (\log\log T)^{k(h+\ell)}$. Hence, 
    \begin{align}\label{28feb.2} 
        \mathcal E_1 
        \ll \max_{1\leq s\leq k} \frac{(\log\log T)^{(h+\ell)/2}}{(\log T)^s} 
        \ll \frac{(\log\log T)^{(h+\ell)/2}}{\log T}
        \ll 1.
    \end{align}
    We now approximate $\log\zeta(\frac{1}{2}+it)$ by $P_x(t)$ on average with respect to the weighted measure $\Starphi(t)^kdt$. We recall that Selberg-Tsang (see \cite[Theorem 5.1]{Tsang1}) proved that for any fixed $m$ and $T^{\delta_1} \le x \le T^{1/(100 m)}$, we have
    \begin{equation}\begin{split}\label{TsangErrorRevised}
        &\frac{1}{T}\int_T^{2 T} (\log|\zeta( \tfrac{1}{2}+it )| - \Re P_x(t))^{2 m} \, dt \ll 1,\\
        &\frac{1}{T}\int_T^{2 T} (\Im\log\zeta( \tfrac{1}{2}+it ) - \Im P_x(t))^{2 m} \, dt \ll 1.
    \end{split}\end{equation}
    This will be enough because the Gaussian moment estimate \cite[Theorem 1.1]{HughesRudnick} in particular implies that under our support condition on $\hat \phi$, the $(k+2)$-th moment of $\Starphi(t)$ is bounded. The approximation then follows from a routine calculation, as we will sketch now.
    Writing $E_x(t) := \log\zeta(\tfrac12+it) - P_x(t)$,
    we have
    \begin{align*}
        \frac{1}{T}\int_T^{2 T} P_x(t)^h \overline{P_x(t)}^\ell \Starphi(t)^k\, dt &= \frac{1}{T}\int_T^{2 T} ( \log \zeta(\tfrac12+ i t))^h \, (\overline{\log\zeta(\tfrac12+it)} )^\ell \Starphi (t)^k \, dt + O(\mathcal E_2),
    \end{align*}
    with
    \begin{align*}
        \mathcal E_2 &\ll \max_{\substack{s\leq h,r\leq \ell \\ s+r\geq 1}}
        \frac{1}{T} \int_T^{2 T}|\log\zeta( \tfrac12+ i t) |^{h-s} |E_x(t)|^s |\log\zeta(\tfrac12+it)|^{\ell-r} |E_x(t)|^r \Starphi(t)^k \, dt\\
        &\ll \max_{1\leq n \leq h+\ell}
        \frac{1}{T} \int_T^{2 T} |\log\zeta( \tfrac12+ i t)|^{h+\ell-n} \Big(|\Re E_x(t)|^{n} + |\Im E_x(t)|^n \Big) \Starphi(t)^k \, dt.
        \end{align*}
        An application of H\"older's inequality and the estimates \eqref{TsangErrorRevised} to bound the resulting moments of $\Re E_x$ and $\Im E_x$ then yields
        \begin{align*}
        \mathcal E_2 
        \ll &\max_{1\leq n \leq h+\ell} \left[ \frac{1}{T} \int_T^{2 T} |\log\zeta( \tfrac12+it)|^{(k+2)(h+\ell-n)} dt \right]^{1/(k+2)}  \left[ \frac{1}{T} \int_T^{2 T} \Starphi(t)^{k+2} dt \right]^{k/(k+2)}.
    \end{align*}
    The last factor here is bounded provided that $\eta<2/(2+k)$, for example by applying \eqref{ExplFormS*} backwards and then using \eqref{HRMomentBound} with the exponent $k$ replaced by $k+2$.
    Applying the classical Selberg's central limit theorem, we get $\mathcal E_2 \ll (\log\log T)^{(h+\ell-1)/2}$ as required.
\end{proof}

As one can see from the previous proof, the repeated application of H\"older's inequality is responsible for the condition $\eta<2/(k+2)$ in \cref{PropApprox}. To obtain full support $(-2/k,2/k)$, which is the natural limit in view of Hughes and Rudnick's result \cite{HughesRudnick}, one needs to be less wasteful when approximating $\log\zeta(\frac12+it)$ by $P_x(t)$ on weighted average. Assuming the Riemann Hypothesis, this can be achieved for the imaginary part of log-zeta by tweaking Selberg's classical argument, as shown by the following result. 

\begin{prop}\label{PropApproxUnderthe Riemann Hypothesis}
    Assume the Riemann Hypothesis. 
    Let $T \ge 3$ be a (large) parameter and let $\ell,k$ be fixed nonnegative integers with $k$ even. Suppose that $\phi$ is a function satisfying \eqref{24may.1} for some $0<\eta<2/k$. Finally, let $0 < \delta_1 < \delta_2 := \frac{1 - \eta k/2}{200 (\ell + 1)}$ and $T^{\delta_1} \le x \le T^{\delta_2}$. Then, we have
    \begin{align*}
        \frac{1}{T} \int_T^{2T} (\Im\log\zeta(\tfrac{1}{2}+it))^\ell (N_\phi(t)-\hat\phi(0))^k dt 
        &= \frac{1}{T}\int_T^{2T} (\Im P_x(t))^\ell \Starphi(t)^k dt \\
        &\quad+O \bigg( \max_{1\leq r\leq \ell} \bigg[\frac{1}{T}\int_T^{2T} (\Im P_x(t))^{2(\ell-r)} \Starphi(t)^k dt\bigg]^{1/2} \bigg).
    \end{align*}
\end{prop}

\begin{rem}\label{Remark_relativeerror}
    The moment estimate that we will prove in \cref{PropMomentPart2alt} implies, in particular, the crude upper bound
    \begin{align*}
        \frac{1}{T}\int_T^{2T} (\Im P_x(t))^{2(\ell-r)} \Starphi(t)^k dt \ll (\log\log T)^{\ell-r}.
    \end{align*}
    This guarantees that the error term in the statement of \cref{PropApproxUnderthe Riemann Hypothesis} is $\ll (\log\log T)^{(\ell-1)/2}$.
\end{rem}

\begin{proof}[Proof of \cref{PropApproxUnderthe Riemann Hypothesis}]
    Following the same argument as in the previous proof (cf. Equations \eqref{28feb.1} and \eqref{28feb.2}), one has
    \begin{align}\label{22may.2}
        \frac{1}{T} \int_T^{2T} (\Im\log\zeta(\tfrac12+it))^\ell (N_\phi(t)-\hat\phi(0))^k dt
        = \frac{1}{T}\int_T^{2T} (\Im\log\zeta(\tfrac12+it))^\ell \Starphi(t)^k dt 
        +O(1).
    \end{align}
    Assume the Riemann Hypothesis and denote $\rho=\frac12+i\gamma$ with $\gamma\in\mathbb R$ the nontrivial zeros of zeta. Then, Selberg-Tsang's classical approximation formula reads (see \cite{Tsang1}, equation (5.15)):
    \begin{align}\label{approx1} 
        \log \zeta(\tfrac12+it)- P_x(t) = A_1(t) +A_2(t) +A_3(t) -L(t) + O(R(t)) 
    \end{align}
    for $T\leq t\leq 2T$ with $t\neq\gamma$, where
    \begin{align*}
        &A_1(t):=\sum_{p\leq x}\big ( p^{-1/2-4/\log x}-p^{-1/2}\big )p^{-it},
        \quad\quad A_2(t):=\sum_{\substack{p^r\leq x \\ r\geq 2}}\frac{p^{-r(1/2+4/\log x+it)}}{r},\\
        &A_3(t):=\sum_{x<n\leq x^3}\frac{\Lambda(n)}{\log n}n^{-1/2-4/\log x-it},
        \quad\quad R(t):= \frac{1}{\log x}\bigg ( \bigg | \sum_{n\leq x^3}\frac{\Lambda (n)}{n^{1/2+4/\log x+it}} \bigg |+\log T \bigg ), \\
        &L(t):=\sum_\rho \int_{1/2}^{1/2+4/\log x}\left ( \frac{1}{2}+\frac{4}{\log x}-u \right )\frac{1}{u+it-\rho}\frac{1}{\frac{1}{2}+\frac{4}{\log x}-\rho}du.
    \end{align*}
    
    We also recall the following estimate from Selberg's work (see \cite[p. 71]{Tsang1}, Equations (5.19)-(5.20) and subsequent display):
    \begin{align}\label{30May.1}
        \Im L(t) &\ll R(t).    
    \end{align}
    Therefore, it suffices to bound the moments of $A_1,A_2,A_3,$ and $R$ with respect to the measure $\Starphi(t)^k dt$. More precisely, we will show that
    \begin{align}\label{23may.1}
    \frac{1}{T}\int_T^{2T} \Big(|A_1(t)|^{2h}+|A_2(t)|^{2h}+|A_3(t)|^{2h}+R(t)^{2h}\Big) \Starphi(t)^{k} dt \ll 1,
    \end{align}
    for every $h\leq 2\ell$.
    Then, upon taking the imaginary part in \eqref{approx1} and plugging it into \eqref{22may.2}, the claim of \cref{PropApproxUnderthe Riemann Hypothesis} follows from the above bound (together with the Cauchy-Schwarz inequality).

    We now embark on the proof of \eqref{23may.1}; for the sake of brevity, we write $k=2m$. We start with $A_1$, for which we write
    \begin{align*}
        \frac{1}{T}\int_T^{2T} |A_1(t)|^{2h} \Starphi(t)^{k}  dt
        &\leq \frac{2^{2m}}{T(\log T)^{2m}}\int_T^{2T} \bigg | A_1(t)^h \bigg( \sum_{n} \frac{\Lambda^*(n)}{n^{1/2+it}} \hat\phi \bigg( \frac{\log n}{\log T} \bigg) \bigg)^m \bigg|^2 dt\\
        &=: \frac{2^{2m}}{T(\log T)^{2m}}\int_T^{2T} \bigg| \sum_{\ell=1}^{\infty}\frac{c_\ell}{\ell^{1/2+it}} \bigg|^{2} dt,
    \end{align*}
    where
    \begin{align*}
        c_{\ell}=\sum_{\substack{p_1\cdots p_h n_1\cdots n_m=\ell \\ p_1, \dots, p_h\leq x\\n_1, \dots, n_m \ge 1 }}\Big(p_1^{-4/\log x}-1 \Big) \cdots \Big(p_h^{-4/\log x}-1 \Big)
    \Lambda^*(n_1) \cdots \Lambda^*(n_m) \hat\phi\bigg(\frac{\log n_1}{\log T}\bigg) \cdots \hat\phi\bigg(\frac{\log n_m}{\log T}\bigg).
    \end{align*}
    Note that, with our choice of $x$, we have $c_\ell=0$ unless $\ell\leq T$. Therefore, an application of the classical mean value theorem for Dirichlet polynomials (see e.g \cite[Chapter 9]{IKowalski}) gives
    \begin{align}
        &\notag\frac{1}{T}\int_T^{2T} |A_1(t)|^{2h} \Starphi(t)^{k}  dt
        \ll \frac{1}{(\log T)^{2m}}  \sum_{\ell=1}^\infty \frac{|c_\ell|^2}{\ell}\\
        &\notag 
        \ll \frac{1}{(\log x)^{h} (\log T)^{2 m}}  \sum_{\substack{p_1,\dots p_h\leq x \\ n_1,\dots n_m \leq T^\eta}} \frac{\log p_1 \cdots \log p_h}{p_1\cdots p_h} 
        \frac{\Lambda^*(n_1) \cdots \Lambda^*(n_m)}{n_1\cdots n_m}  \times \\ &\notag \qquad \qquad \qquad \times \sum_{\substack{ \tilde n_1,\dots,\tilde n_m \leq T^\eta \\ \tilde p_1,\dots,\tilde p_h \leq x}} \Lambda^*(\tilde n_1) \cdots \Lambda^*(\tilde n_m) \times \mathbbm{1}(p_1\cdots p_hn_1\cdots n_m = \tilde p_1\cdots \tilde p_h \tilde n_1\cdots \tilde n_m),
    \end{align}
    by employing $1-p_i^{-4/\log x} \ll \log p_i/\log x$ and $1-\tilde p_i^{-4/\log x} \ll1$. Since $p_1,\dots,p_h,n_1,\dots,n_m$ are primes or squares of primes,  there are only a bounded number of $\tilde p_1,\cdots, \tilde p_h, \tilde n_1,\cdots, \tilde n_m$ such that the equality $p_1\cdots p_hn_1\cdots n_m = \tilde p_1\cdots \tilde p_h \tilde n_1\cdots \tilde n_m$ holds. As a consequence, using the bound $\Lambda^*(\tilde n_j) \ll \log T$, we see that the innermost sum satisfies
    \[ \sum_{\substack{ \tilde n_1,\dots,\tilde n_m \leq T^\eta \\ \tilde p_1,\dots,\tilde p_h \leq x}} \Lambda^*(\tilde n_1) \cdots \Lambda^*(\tilde n_m) \times \mathbbm{1}(p_1\cdots p_hn_1\cdots n_m = \tilde p_1\cdots \tilde p_h \tilde n_1\cdots \tilde n_m) \ll (\log T)^m.\]
    Therefore, we finally obtain
    \begin{align*}
        &\notag\frac{1}{T}\int_T^{2T} |A_1(t)|^{2h} \Starphi(t)^{k}  dt
        \ll \frac{1}{(\log x)^h (\log T)^{m}}  \bigg(\sum_{p\leq x } \frac{\log p}{p} \bigg)^h \bigg( \sum_{n \leq T^\eta}
        \frac{\Lambda(n)}{n}\bigg)^m 
        \ll 1,
    \end{align*}
    as desired.

    We now move to $A_2$. Since
    \begin{align*}
        A_2(t)=\frac{1}{2}\sum_{\substack{q\leq x\\ q=p^2, \; p \text{ prime}}} \frac{q^{-4/\log x}} {q^{1/2+it}}+ O(1),
    \end{align*}
    to show that the $2h$-th weighted moment of $A_2$ is bounded, it suffices to show that
    \begin{align}\label{27may.1}
        \frac{1}{T(\log T)^{k}} \int_T^{2T} \bigg|\sum_{\substack{q\leq x\\ q=p^2, \; p \text{ prime}}} \frac{q^{-4/\log x}} {q^{1/2+it}} \bigg|^{2h} 
        \bigg| \sum_{n} \frac{\Lambda^*(n)}{n^{1/2+it}} \hat\phi \bigg(\frac{\log n}{\log T} \bigg) \bigg|^{k} dt 
        \ll 1. 
    \end{align}
    The bound $\frac{1}{T}\int_T^{2T} |A_2(t)|^{2h} \Starphi(t)^kdt \ll 1$ then follows by Cauchy-Schwarz.
    The bound in \eqref{27may.1} can be proven with the same strategy used for $|A_1|$. Specifically, denoting
    \begin{align*}
        d_\ell:=
        \sum_{\substack{ n_1,\dots,n_m\leq T^\eta \\ p_1,\dots,p_h\leq \sqrt{x} \\ p_1^2 \dots p_h^2 n_1\dots n_m=\ell}}
        \Lambda^*(n_1)\cdots \Lambda^*(n_m) p_1^{-8/\log x} \cdots p_h^{-8/\log x},
    \end{align*}
    by an application of the mean value theorem, the left-hand side of \eqref{27may.1} can be bounded by (again, we write $k=2m$)
    \begin{align*}
        &\ll \frac{1}{(\log T)^{2m}} \sum_{\ell} \frac{|d_\ell|^2}{\ell}
        \leq \frac{1}{(\log T)^{2m}} \sum_{\ell} \frac{1}{\ell} \bigg( \sum_{\substack{ n_1,\dots,n_m\leq T^\eta \\ p_1,\dots,p_h\leq \sqrt{x} \\ p_1^2 \dots p_h^2 n_1\dots n_m=\ell}}\Lambda^*(n_1)\cdots \Lambda^*(n_m)\bigg)^2 \\
        &\ll \frac{1}{(\log T)^{m}} \bigg(\sum_{p \leq \sqrt{x} } \frac{1}{p^2}\bigg)^h \bigg( \sum_{n\leq T^\eta} \frac{\Lambda(n)}{n}\bigg)^{m} \ll 1.
    \end{align*}
    Note that in the third step above, we employed the same prime divisor argument used for $A_1$.

    By a similar calculation, one can also easily bound the weighted moments of $A_3$ and $R$. Indeed, noting that both the support of the Dirichlet polynomials in the definitions of $A_3$ and $R$ can be truncated to primes and squares of primes up to a $O(1)$ error, one has
    \begin{equation*}\begin{split}
        \frac{1}{T}\int_T^{2T} |A_3(t)|^{2h} \Starphi(t)^k &\ll  \bigg( \sum_{x<n<x^3} \frac{\Lambda(n)}{n\log n} \bigg)^h \bigg( \frac{1}{\log T} \sum_{n\leq T^\eta} \frac{\Lambda(n)}{n} \bigg)^{k/2} \ll 1, \\
        \frac{1}{T}\int_T^{2T} |R(t)|^{2h} \Starphi(t)^k &\ll \bigg( \frac{1}{\log x} \sum_{n<x^3} \frac{\Lambda(n)}{n} \bigg)^{2h} \bigg(\frac{1}{\log T} \sum_{n\leq T^\eta} \frac{\Lambda(n)}{n} \bigg)^{k/2} \ll 1.
    \end{split}\end{equation*}
\end{proof}

\begin{rem}
    In the proof of \cref{PropApproxUnderthe Riemann Hypothesis}, the Riemann Hypothesis is invoked in a somewhat subtle way to prove \eqref{23may.1}. Unconditionally, the terms $A_1,A_2,A_3,$ and $R$ on the right-hand side of \eqref{approx1} are not Dirichlet polynomials, as the coefficients depend (mildly) on $t$. As a result, the classical mean value theorem does not seem sufficient to analyze their weighted moments, and some further information on the behavior of $\Starphi$ in short intervals appears necessary.
\end{rem}

\begin{rem}\label{RemarkImvsRe}
    The same proof presents complications when applied to $\log|\zeta(\frac12+it)|$. The difficulty arises in the analogue of \eqref{30May.1} for $\Re L(t)$; see \cite{Tsang1}, Equations (5.19)-(5.21) and subsequent paragraphs, up to the first display on page 74. Assuming the Riemann Hypothesis, we have the following bound for the real part:
    \begin{align*}
        \Re L(t) &\ll R(t)\bigg( 1 + \log^+\bigg(\frac{1}{\nu_t\log x}\bigg)\bigg),
    \end{align*}
    where $\nu_t:= \min_\gamma|\gamma-t|$ and $\log^+(x)=\max(0,\log x)$. When bounding the moments of $\Re L(t)$ with respect to the measure $\Starphi(t)^k dt$, the dependence on the nontrivial zeros via $\nu_t$ seems to require information about the $(k+1)$-correlation of the zeros, rather than just the $k$-correlation. Indeed, one can in fact invoke work of Rudnick-Sarnak \cite{RudnickSarnak} to bound the contribution involving $\log^+(1/\nu_t \log x)$, even unconditionally when looking for an order of magnitude upper bound. But this introduces an additional zero and hence relies on $(k+1)$-correlations, requiring $\eta < 2/(k+1)$.
\end{rem}

\section{Moment estimates}\label{SectionMomentEstimates}

In this section, we compute the mixed moments of the Dirichlet polynomials that approximate $\log\zeta(\frac12+it)$ and $\Starphi(t)$; see Equation \eqref{4june.1} and \cref{ExplForm_SumOverZeros} for the definitions. For the reader's convenience, we split the proof into two parts. In \cref{prop:lkmoments}, we analyze the integral over $t$, and show that the main term comes from the \lq\lq diagonal contribution\rq\rq, cf. Equation \eqref{eq:MhkDiag}.

\begin{prop}\label{prop:lkmoments}
    Let $T \ge 3$ be a (large) parameter and let $h,\ell,k$ be fixed nonnegative integers. Set $x = T^{\varepsilon}$ for some sufficiently small constant $\varepsilon = \varepsilon(h,\ell,k,\eta) > 0$ and define 
    \begin{align*}
           M(h,\ell,k) = M(h, \ell,k; T) := \frac{1}{T}\int_T^{2 T} P_x(t)^{h} \overline{P_x(t)}^{\ell} \Starphi(t)^k\, dt.
    \end{align*}
    Then, for some sufficiently small constant $\delta>0$, we have
    \begin{equation}\begin{split}\label{eq:MhkDiag}
        M(h,\ell, k) = \frac{(-1)^k}{(\log T)^k} 
        \sum_{J \subseteq \{1, \dots, k \}} \sum_{\substack{p_1, \dots, p_h\leq x\\
        q_1, \dots, q_\ell\leq x\\ n_1, \dots, n_k \ge 1 }} &\frac{\Lambda^*(n_1) \cdots \Lambda^*(n_k)}{\sqrt{p_1\cdots p_h q_1\dots q_\ell n_1 \cdots n_k}} \prod_{j=1}^k  \hat\phi\left(\frac{\log n_j}{\log T}\right) \times \\
        &\times\mathbbm{1} \bigg( \prod_{i\leq h} p_i \prod_{j \in J} n_j = \prod_{i\leq \ell} q_i \prod_{j\in J^c} n_j \bigg) + O (T^{-\delta}).
    \end{split}\end{equation}
\end{prop}

\begin{proof}
    By definition of $P_x$ and $\Starphi$, writing cosines as exponentials and multiplying out, we get
    \begin{equation}\begin{split}\label{eq:MhkDef}
        M(h,\ell, k) = \frac{(-1)^k}{(\log T)^k} \sum_{J \subseteq \{1, \dots, k \}} &\sum_{\substack{p_1, \dots, p_h  \leq x\\
        q_1, \dots, q_\ell\leq x\\
        n_1, \dots, n_k \ge 1}} \frac{\Lambda^*(n_1) \cdots \Lambda^*(n_k)}{\sqrt{p_1\cdots p_h q_1\dots q_\ell n_1 \cdots n_k }} \times \\ \times &\prod_{j=1}^k \hat\phi\left(\frac{\log n_j}{\log T}\right) \frac{1}{T} \int_T^{2 T} \left( \frac{\prod_{i\leq \ell} q_i \prod_{j \in J^c} n_j}{\prod_{i\leq h} p_i \prod_{j \in J} n_j} \right)^{i t} \, dt.
    \end{split}\end{equation}
    To deduce the claim, we first recall the standard bound
    \begin{equation}\label{eq:xitBound}
        \int_T^{2T} x^{it} \, dt \ll \frac{1}{| \log x |}
    \end{equation} 
    for $x \ne 1$, which follows immediately from integrating.
    Let us fix $J\subseteq\{1, \dots, k\}$  and denote $$P(J):=\frac{\prod_{i\leq \ell} q_i \prod_{j \in J^c} n_j}{\prod_{i\leq h} p_i \prod_{j \in J} n_j}.$$ 
    We split the range of summation in \eqref{eq:MhkDef} according to whether
    \begin{equation}\label{eq:QuotNear1}
        \tfrac{1}{2} \le P(J)  \le 2
    \end{equation}
    holds. 
    
    If the inequality \eqref{eq:QuotNear1} does not hold, then by \eqref{eq:xitBound} the integral appearing in \eqref{eq:MhkDef} is bounded. The contribution from such tuples to \eqref{eq:MhkDef} is thus
    \begin{align*}
        \ll \frac{1}{T} \sum_{J \subseteq \{1, \dots, k \}} \sum_{\substack{p_1, \dots, p_h\leq x \\
        q_1, \dots, q_\ell\leq x\\
        n_1, \dots, n_k \le T^\eta
        }} \frac{1}{\sqrt{p_1\cdots p_h  q_1\dots q_\ell n_1 \cdots n_k}} 
        \ll \frac{x^{(h+\ell)/2}T^{\eta k/2}}{T}
        \ll T^{-\delta},
    \end{align*}
    for some fixed small $\delta>0$.
    We can hence restrict the sum in \eqref{eq:MhkDef} to tuples satisfying \eqref{eq:QuotNear1}, up to an error of $\ll T^{-\delta}$. Clearly, the case $P(J)=1$
    corresponds to the main term in \eqref{eq:MhkDiag}. 
    
    For the remaining range $P(J)\in[\frac12,2] \setminus \{ 1 \}$, employing \eqref{eq:xitBound}, it thus suffices to show that, for any $J \subseteq \{1,\dots,k\}$,
    \begin{equation}\label{eq:QuotNear1STS} 
        \sum_{\substack{p_1, \dots, p_h  \leq x\\ q_1, \dots, q_\ell\leq x \\ n_1, \dots, n_k \le T^\eta}} 
        \frac{1}{\sqrt{p_1\dots p_h q_1\dots q_\ell n_1\dots n_k}} \frac{\mathbbm{1} \left( P(J) \in \left[ \tfrac{1}{2}, 2 \right] \setminus \{1 \}  \right)}{\left| \log P(J) \right|}  
        \ll T^{1-\delta}.
    \end{equation}
    Note that we may assume without loss of generality that $|J| \ge \frac{k}{2}$; otherwise, one can simply exchange $J$ with $J^c$ and swap the roles of $q_i$ and $p_i$ in the following argument. 
    Now, writing
    \begin{align}\label{eq:shiftbyh}
        \prod_{i\leq h} p_i \prod_{j\in J} n_j=\prod_{i\leq \ell} q_i \prod_{j\in J^c} n_j +h
    \end{align}
    for some nonzero integer $h$ with $-T\leq h\leq T$, we have
    \begin{align*}
        \frac{1}{\vert \log P(J)\vert}\ll \frac{\prod_{i\leq h} p_i \prod_{j\in J} n_j}{\vert h\vert}.
    \end{align*}
    As a consequence, we deduce
    \begin{align*}
        \frac{1}{\sqrt{p_1\dots p_h q_1\dots q_\ell n_1\dots n_k}} \frac{1}{\left| \log P(J) \right|} 
        \ll  \frac{\prod_{j\leq h} p_i \prod_{j\in J} n_j}{|h|\sqrt{p_1\dots p_h q_1\dots q_\ell n_1\dots n_k}} 
        = \frac{1}{|h|\sqrt{P(J)}} 
        \ll \frac{1}{|h|}  ,
    \end{align*}
    by \eqref{eq:QuotNear1}. Thus, we have that the left-hand side of \eqref{eq:QuotNear1STS} is bounded by
    \begin{align*}
        &\ll \sum_{0\neq |h| \leq T} \frac{1}{|h|} \sum_{\substack{p_1, \dots, p_h  \leq x\\ q_1, \dots, q_\ell\leq x \\ n_1, \dots, n_k  \leq T^\eta}} \mathbbm{1} \left( \prod_{i\leq h} p_i \prod_{j\in J} n_j = \prod_{i\leq \ell} q_i \prod_{j\in J^c} n_j+h \right)\\
        &\ll \sum_{\substack{q_1,\dots,q_\ell \leq x \\ n_j\leq T^\eta,\, j\in J^c}}  \sum_{0\neq |h| \leq T} \frac{1}{|h|} \sum_{\substack{p_1, \dots, p_h  \leq x \\ n_j \leq T^\eta, \, j\in J}} \mathbbm{1} \left( \prod_{i\leq h} p_i \prod_{j\in J} n_j = \prod_{i\leq \ell} q_i \prod_{j\in J^c} n_j+h \right).
    \end{align*}
    Now, for any given $(q_i)_{i\leq \ell}$, $(n_j)_{j\in J^c}$, and $h$, the number of possible $(p_i)_{i\leq h},(n_j)_{j\in J}$ such that the indicator function above equals 1 is certainly $\ll T^{\varepsilon}$, say. Thus, the last line in the above display is bounded by
    \begin{align*}
        &\ll T^{\varepsilon} \sum_{\substack{q_1,\dots,q_\ell \leq x \\ n_j\leq T^\eta,\, j\in J^c}}  \sum_{0\neq |h| \leq T} \frac{1}{|h|}
        \ll T^{2\varepsilon}  x^{\ell} T^{\eta |J^c|} 
        \ll T^{(\ell + 2)\varepsilon} T^{\eta |J^c|}.
    \end{align*}
    Since $|J^c|\leq \frac k2$ by assumption, taking $\varepsilon$ small enough, we find that the above is $\ll T^{1-\delta}$. This concludes the proof of \eqref{eq:QuotNear1STS}, and therefore that of \eqref{eq:MhkDiag}.
\end{proof}

As a final step, we analyze the combinatorics of the main term in \cref{prop:lkmoments} and show that the moments $M(h,\ell,k)$ are indeed those of a Gaussian.

\begin{prop}\label{PropMomentPart2alt}
    Let $T \ge 3$ be a (large) parameter and let $h,\ell,k$ be fixed nonnegative integers. Set $x = T^{\varepsilon}$ for some sufficiently small constant $\varepsilon = \varepsilon(h,\ell,k,\eta) > 0$. Then we have
    \begin{equation}\begin{split}\label{MhkDiagEvalAlt}
     M(h,\ell, k) = \mu_k \sigma_\phi^k \mathbbm{1}(h = \ell) h! \left(\log\log T\right)^{h} + O \left( (\log\log T)^{(h + \ell -1)/2} \right).
    \end{split}\end{equation}
\end{prop}

\begin{proof}
By \cref{prop:lkmoments}, it suffices to show that
\begin{equation}\begin{split}\label{MhkdiagSTS}
    \frac{(-1)^k}{(\log T)^k} 
        \sum_{\substack{J\subseteq \{1, \dots, k\}\\p_1, \dots, p_h\leq x\\
        q_1, \dots, q_\ell\leq x}} &\frac{\Lambda^*(n_1) \cdots \Lambda^*(n_k)}{\sqrt{p_1\cdots p_h q_1\dots q_\ell n_1 \cdots n_k}} \prod_{j=1}^k  \hat\phi\left(\frac{\log n_j}{\log T}\right) \mathbbm{1} \left( \prod_{i\leq h} p_i \prod_{j \in J} n_j = \prod_{i\leq \ell} q_i \prod_{j\in J^c} n_j \right)\\
        &= \mu_k \sigma_\phi^k \mathbbm{1}(h = \ell) h! \left( \log\log T \right)^{h} + O \left( (\log\log T)^{(h + \ell -1)/2} \right).
\end{split}\end{equation}
    Recall that we may assume all the $n_j$ to be either be primes or their squares. From the equation 
    \[ \prod_{i\leq h} p_i \prod_{j \in J} n_j = \prod_{i\leq \ell} q_i \prod_{j\in J^c} n_j  \]
    it therefore follows that each of the variables has to be in one of the relations
    \begin{equation}\label{2PR}
        p_i = q_{i'},\quad n_j = q_i, \quad p_i = n_j, \quad n_j = n_{j'} \qquad (j\in J, j'\in J^c)
    \end{equation}
    involving two prime powers, or one of the relations
    \begin{equation}\label{3PR}
        p_i p_{i'} = n_j, \quad p_i n_j = n_{j'}, \quad n_j n_{j'} = n_{j''}, \quad n_j = q_i q_{i'}, \quad n_j = q_i n_{j'}, \quad n_j = n_{j'} n_{j''}
    \end{equation}
    involving three prime powers. Here, a variable $n_j$ on the left-hand side refers to an index $j \in J$ and on the right-hand side it refers to an index $j \in J^c$. 
    As a consequence, for any choice of $J$, the left-hand side of \eqref{MhkdiagSTS} factors as a product of the individual sums corresponding to each of the $10$ relations \eqref{2PR} and \eqref{3PR}, which we refer to in their order by $S_1,\dots, S_{10}$. For example, the sum corresponding to the equality $p_i = n_j =: p$ would be
    \[ S_3 = -\frac{1}{\log T} \sum_{p \le x} \frac{\Lambda(p)}{p} \hat \phi \left( \frac{\log p}{\log T} \right). \]
    Note that we incorporate a factor $-1/\log T$ in the sum for each occurrence of a variable $n_j$ in the relation.
    
    Let us denote further by $u:= (u_1,\dots,u_{10})$ the number of occurrences of each of the $10$ relations.
    We can therefore write the left-hand side of \eqref{MhkdiagSTS} as
    \[ = \sum_{u,J} C(u,J) \prod_{i=1}^{10} S_i^{u_i}, \]
    where the $C(u,J)$ count the number of ways to select indices with the given number of relations $u$ of each type. As we will see, there is no need to figure out the values of these constants exactly except in the special case when the $p_i$ pair up perfectly with the $q_i$ and the $n_j$ pair up perfectly among themselves, or in other words, when $u_i = 0$ for all $i$ except $i = 1$ and $i = 4$.

    One verifies through elementary summation by parts arguments involving the Prime Number Theorem that we have
    \begin{equation}\begin{split}
        S_1 &= \sum_{p \le x} \frac{1}{p} = \log \log T + O(1), \\
        S_4 &= \frac{1}{(\log T)^2} \sum_{n \ge 1} \frac{\Lambda^*(n)^2}{n} \hat \phi^2 \left( \frac{\log n}{\log T} \right) = \frac{\sigma_\phi^2}{2} + O \left( \frac{1}{\log T} \right),\\
        S_2, S_3 &\ll 1 \qquad \text{ and } \qquad S_i \ll \frac{1}{\log T} \qquad (i = 5,\dots, 10).
    \end{split}\end{equation}

    We can thus deduce that the left-hand side of \eqref{MhkdiagSTS} is $\ll (\log \log T)^{u_1}$, and we must evidently have $u_1 \le \min \{h,\ell\}$ since it is the number of relations $p_i = q_j$. This can be absorbed into the error term unless $h = \ell = u_1$, which we assume from now on. As a consequence, the $n_j$ must mix solely among themselves, or in other words, we must have $u_2 = u_3 = u_5 = u_6 = u_8 = u_9 = 0$. From this, we deduce that the left-hand side of \eqref{MhkdiagSTS} is in fact
    \[ \ll \frac{(\log \log T)^h}{(\log T)^{u_7+u_{10}}}, \]
    which may again be absorbed into the error term unless $u_7 = u_{10} = 0$. This forces $u_4 = k/2$, from which we can further assume $k$ to be even. Hence, the $n_j$ variables have to match perfectly among themselves, and we must have $|J| = k/2$.

    Now the number of ways of selecting indices to satisfy the relations $p_i = q_j$ is simply $h! = \ell!$. The number of matchings of sets $J$ and $J^c$ is $\frac{k!}{(k/2)!}$. 
    Thus, we may conclude that the left-hand side of \eqref{MhkdiagSTS} is in fact
    \begin{align*} &= h! S_1^h \frac{k!}{(k/2)!} S_4^{k/2} + O \left( (\log\log T)^{(h + \ell -1)/2} \right) \\ 
    &= \mu_k \sigma_\phi^k \mathbbm{1}(h = \ell) h! \left(\log\log T\right)^{h} + O \left( (\log\log T)^{(h + \ell -1)/2} \right).
    \end{align*}
\end{proof}

\begin{rem}\label{RemarkHRsharp}
    If $h = \ell = 0$, then the only relations appearing in \eqref{2PR} and \eqref{3PR} are the ones in the $n_j$ variables. As a consequence, the main term continues to correspond to the case where the $n_j$ pair up perfectly, and has the same shape. As for the error term, one only has contributions from $u_7 > 0$ and $u_{10} > 0$, and the corresponding sums are both $\ll (\log T)^{-1}$. Hence, one deduces from an inspection of the proof that for any $k$ we have
    \begin{equation}
        M(0,0,k) = \frac{1}{T} \int_T^{2 T} S_\phi^*(t)^k \, dt = \mu_k \sigma_\phi^k + O \left( \frac{1}{\log T} \right). 
    \end{equation}
    This recovers the result of Hughes-Rudnick \cite{HughesRudnick} with the sharp cutoff $\omega = \mathbbm{1}_{[0,1]}$, recalling that switching between $S_\phi^*$ and $S_\phi$ is immediate by \eqref{SS*Comp}.
\end{rem}

\subsection*{Proof of \texorpdfstring{\cref{mainthm}}{mainthm} and \texorpdfstring{\cref{mainthm2}}{mainthm2} }

\cref{mainthm} follows immediately by combining \cref{PropApprox} and \cref{PropMomentPart2alt}. We now assume the Riemann Hypothesis and prove \cref{mainthm2}. Using \cref{PropApproxUnderthe Riemann Hypothesis} and \cref{PropMomentPart2alt} in the form of \cref{Remark_relativeerror}, one has
\begin{align*}
    \frac{1}{T} \int_T^{2T} (\Im\log\zeta(\tfrac{1}{2}+it))^\ell (N_\phi(t)-\hat\phi(0))^k dt 
        &= \frac{1}{T}\int_T^{2T} (\Im P_x(t))^\ell \Starphi(t)^k dt 
        +O\left( (\log\log T)^{(\ell-1)/2} \right).
\end{align*}
By \cref{PropMomentPart2alt}, the main term above is
\begin{align*}
    \frac{1}{T}\int_T^{2T} (\Im P_x(t))^\ell \Starphi(t)^k dt 
    &= \frac{1}{(2i)^\ell} \sum_{a+b=\ell} \frac{\ell! (-1)^b}{a! b!} M(a,b,k) \\
    &= \frac{\mu_k \sigma_\phi^k}{(2i)^\ell} \sum_{a+b=\ell} \frac{\ell! (-1)^b}{b!}  \mathbbm{1}(a = b) \left(\log\log T\right)^{a} + O \left( (\log\log T)^{(\ell -1)/2} \right) \\
    &= \mu_k \sigma_\phi^k \frac{\ell!}{2^\ell(\ell/2)!}  \mathbbm{1}(\ell \text{ even})  \left(\log\log T\right)^{\ell/2} + O \left( (\log\log T)^{(\ell -1)/2} \right) 
\end{align*}
and \cref{mainthm2} is proven.
\qed

\subsection*{Proof of \texorpdfstring{\cref{cor1}}{cor1} and \texorpdfstring{\cref{cor2}}{cor2} }

Both the results follow from an application of the method of moments \cite[Section 30]{Billingsley}, see also \cite[Section B.5]{Kowalski1}. By \eqref{24may.1} and \cref{ExplForm_SumOverZeros}, one sees that $N_\phi(t) - \hat\phi(0)$ is approximately real; the error term is of size $O(1/\log T)$ and can therefore be neglected. Since $k$ is even, the probability measure $d\mu_{\phi,k}(t)$ is then well-approximated by $(N_\phi(t) - \hat\phi(0))^k \frac{dt}{C_{\phi,k}}$. Therefore, \cref{cor1} (respectively, \cref{cor2}) follows from the moment estimates given by \cref{mainthm} (respectively, \cref{mainthm2}), since $C_{\phi,k}\sim \mu_k \sigma_\phi^k$ by \cref{RemarkHRsharp}.
\qed

\bibliographystyle{plainnat}
\bibliography{bibliography}

\end{document}